\documentclass[11pt]{article}
\usepackage{graphicx}
\usepackage{amsmath}
\usepackage{multirow}
\usepackage{amssymb}
\usepackage{amsthm}
\usepackage{booktabs}
\usepackage{bm}
\usepackage{diagbox}
\usepackage{natbib} 
\usepackage{moreverb}
\usepackage{geometry}
\usepackage{bbm}

\usepackage{setspace}
\def\T{{ \mathrm{\scriptscriptstyle T} }}

\def\pr{{\operatorname{pr}}}

\theoremstyle{plain}

\newtheorem{theorem}{Theorem}[section]
\newtheorem{lemma}[theorem]{Lemma}
\newtheorem{proposition}[theorem]{Proposition}
\newtheorem{corollary}[theorem]{Corollary}
\theoremstyle{definition}

\newtheorem*{remark}{Remark}

\title{Gaussian comparison above the median}
\author{Colin B. Fogarty \thanks{Department of Statistics, University of Michigan, Ann Arbor, MI.}}
\date{}
\begin{document}
\thispagestyle{empty}
\maketitle
\abstract{We prove a Gaussian comparison inequality for closed convex sets with reference probability at least 1/2.  For centered Gaussian vectors whose covariance matrices are ordered in the Loewner sense, the smaller covariance law assigns at least as much probability as the larger covariance law to every closed convex set with measure at least 1/2 under the larger covariance.  More generally, we show that convolving a nondegenerate Gaussian law with any independent centrally symmetric distribution cannot increase the probability of a closed convex set when the convolved probability is at least 1/2.  This provides a one-sided analogue of Anderson's Theorem for Gaussian measures. As a statistical application, the result justifies one-sided and order-restricted inference using conservative covariance estimators at significance levels below 1/2 for test statistics whose acceptance regions are closed and convex but not necessarily symmetric.}
\normalsize
\singlespacing
\section{Introduction}
\subsection{Anderson's Theorem and conservative two-sided inference}

In many design-based and assumption-lean statistical settings it is challenging, or even impossible, to construct a consistent estimator for the covariance of a root-$n$ consistent and asymptotically normal estimator $\hat{\theta} \in \mathbb{R}^d$. One prominent example arises in completely randomized experiments analyzed under the finite population model, where the true covariance for the treated-minus-control difference in means depends on the unobservable treatment effects \citep{ney23, li17, aba20}. Similar issues arise in randomized experiments under interference \citep{har26}; in non-measurable sampling designs where pairwise inclusion probabilities can equal zero \citep[p. 33]{aro13conservative, sar92}; in inference for misspecified regression models with fixed covariates \citep{aba14}; and more broadly when observations are independent but not identically distributed \citep{liu95}. In such settings, a common expedient is to develop a {conservative} covariance estimator: if $\Sigma_X = \text{avar}\{n^{1/2}(\hat{\theta}-\theta)\}$, one constructs an estimator $\hat{\Sigma}_Y$ converging in probability to $\Sigma_Y \succeq \Sigma_X$. 

Anderson's Theorem \citep{and55} converts conservative covariance estimation into conservative two-sided inference. Let $X = (X_1,\ldots,X_d)^\T$ and $Y = (Y_1,\ldots,Y_d)^\T$ be mean-zero Gaussian random vectors with covariance matrices $\Sigma_X$ and $\Sigma_Y$ satisfying $\Sigma_Y-\Sigma_X \succeq 0$. \citet{and55} proves that for any closed convex set $\mathcal{A}$ that is mirror symmetric about the origin (i.e., $\mathcal{A}=-\mathcal{A}$), $\pr(X\in \mathcal{A})\geq \pr(Y\in \mathcal{A}).$ Consequently, for any $f: \mathbb{R}^d\rightarrow \mathbb{R}$ that is lower semicontinuous, quasiconvex and mirror symmetric, $\pr\{f(X)\geq a\} \leq \pr\{f(Y)\geq a\}$ for all $a$, establishing that $f(Y)$ stochastically dominates $f(X)$. When $d=1$ with $X\sim \mathcal{N}(0,\sigma^2_X)$, $Y\sim \mathcal{N}(0,\sigma^2_Y)$ and $\sigma^2_Y\geq \sigma^2_X$, taking $f(w) = |w|$ yields asymptotically conservative two-sided confidence intervals $\hat{\theta} \pm \Phi^{-1}(1-\alpha/2)\hat{\sigma}_Y/n^{1/2}$, and asymptotically conservative two-sided $p$-values $2\Phi\{-n^{1/2}|\hat{\theta}-\theta_0|/\hat{\sigma}_Y\}$, where $\Phi$ and $\Phi^{-1}$ are the cumulative distribution function and quantile function of the standard normal distribution. For $d > 1$, the result provides conservative two-sided inference using roots such as $f(w) = \max_{1\leq i \leq d}|w_i|$, $f(w) = ||w||^2_2$, and $f(w) = w^\T Vw$ for any $V \succeq 0$.
\subsection{Towards conservative one-sided inference}
The restriction to mirror symmetric sets precludes the use of Anderson's theorem for one-sided or order-restricted hypothesis testing \citep{sil05}. With the assumption of mirror symmetry dropped, full stochastic dominance no longer holds. For example, in the univariate setting with $\sigma^2_Y\geq \sigma^2_X>0$,  $\pr(X\geq a) \geq \pr(Y \geq a)$ for $a < 0$, while $\pr(X\geq a) \leq \pr(Y \geq a)$ for $a \geq 0$. From this univariate example, we see that when testing the one-sided null $\theta \leq \theta_0$ with alternative $\theta > \theta_0$, conservative covariance estimation does not justify conservative inference for all significance levels $\alpha$. Nevertheless, it {does} justify conservative inference for all $\alpha \leq 1/2$. As typical values of $\alpha$ rarely exceed 0.1 in practice, the loss of full stochastic dominance and the restriction to conservative inference for $\alpha \leq 1/2$ has little bearing on the applicability of the result.

For $d > 1$, relevant functions $f$ for one-sided testing include $f(w) = \max_{1\leq i\leq d} w_i$, $f(w) = \sum_{i=1}^d\max(0, w_i)^2$, and $f(w) = ||\Pi_{\mathcal{C}}(w)||$, the norm of the projection of $w$ onto a closed convex cone $\mathcal{C}$. Anderson's Theorem cannot generally be applied to these functions as they are not necessarily mirror symmetric. Replacing the covariance condition in Anderson's Theorem with the weaker assumption that $E\{(Y_i-Y_j)^2\} \geq E\{(X_i-X_j)^2\}$ for all $i,j = 1,\ldots,d$, the Sudakov-Fernique inequality states that $E(\max_{1\leq i\leq d} X_i) \leq E(\max_{1\leq i\leq d} Y_i)$ for centered Gaussians \citep[Theorem 2.2.3]{adl07}. Further imposing $E(Y_i^2) = E(X_i^2)$ for all $i=1,\ldots,d$ in addition to the Sudakov-Fernique hypotheses, Slepian's inequality states that $\pr(\max_{1\leq i\leq d} X_i \geq a) \leq \pr(\max_{1\leq i\leq d} Y_i \geq a)$ for all $a$ \citep[Theorem 2.2.1]{sle62,adl07}. Sudakov-Fernique does not provide control of tail probabilities, and the additional assumption of equal marginal variances makes the covariance conditions of Slepian's inequality incomparable with those of Anderson's Theorem. These inequalities do not confer conservative inference under conservative covariance estimation even for the one-sided max statistic for a range of $\alpha$ relevant for statistical inference, let alone the broader class of statistics encountered in multivariate one-sided testing. 

\citet[\S\S 3.2-3.3]{ehm18} conjectured that under the Loewner ordering $\Sigma_Y\succeq \Sigma_X$ of Anderson's Theorem, $\pr\{f(X) >a\} \leq \pr\{f(Y) > a\}$ for any convex monotone function $f$ whenever $ \pr\{f(Y) > a\} \leq 1/2$. \citet[\S 9.1]{coh22} conjectured that if $Y$ has larger marginal variances than $X$ and the Sudakov-Fernique covariance conditions hold, $\pr(\max_{1\leq i\leq d} X_i \geq a) \leq \pr(\max_{1\leq i\leq d} Y_i \geq a)$ whenever $\pr(\max_{1\leq i\leq d} Y_i \geq a) \leq 1/2$.  

\section{Comparison under symmetric convolution}
We first state a result comparing the probability content of a closed convex set under a nondegenerate Gaussian measure to that of a convolution of the same nondegenerate Gaussian measure with the law of a centrally symmetric random vector. 

\begin{theorem}\label{thm:conv}
Let $X\in \mathbb{R}^d$ be a mean-zero Gaussian random vector with covariance $\Sigma_X\succ 0$, and let $W \in \mathbb{R}^d$ be a random vector independent of $X$ satisfying $W \overset{d}{=}-W$. For every closed convex set $\mathcal{K}$,
\begin{align*}
\pr(X+W \in \mathcal{K})\geq \frac{1}{2} \Rightarrow \pr(X \in \mathcal{K})\geq \pr(X+W\in \mathcal{K}).
\end{align*}
\end{theorem}
The proof is presented in \S \ref{sec:proof}. As a core ingredient, it uses the Ehrhard--Borell inequality \citep{ehr83, bor03} to establish the concavity of $H(v) = \Phi^{-1}\{\pr(X \in \mathcal{K}+v)\}$ in $v$ when $0 < \pr(X\in \mathcal{K} + v) < 1$ for all $v$. It then shows that the requirement $\pr(X+W \in \mathcal{K})\geq 1/2$ implies that $\pr(X+v\in \mathcal{K}) + \pr(X-v\in\mathcal{K})\leq 2\;\pr(X\in \mathcal{K})$ for any $v$, which yields the result.

As a corollary, we obtain a one-sided analogue of Anderson's Theorem for Gaussian measures comparing the probability content of closed convex sets under the Loewner ordering. 
\begin{corollary}\label{cor:compare}
 Let $X\in \mathbb{R}^d$  and $Y\in \mathbb{R}^d$ be mean-zero Gaussian random vectors with covariance matrices $\Sigma_X$ and $\Sigma_Y$ satisfying $\Sigma_Y-\Sigma_X \succeq 0$. Let $\mathcal{K}$ be a closed convex set. Then, \begin{align*}\pr(Y\in\mathcal{K})\geq \frac{1}{2}\Rightarrow \pr(X\in \mathcal{K}) \geq \pr(Y\in \mathcal{K}).\end{align*}
\end{corollary}

As lower semicontinuous quasiconvex functions have closed convex sublevel sets, we additionally obtain the following result comparing tail probabilities above the median.
\begin{corollary}\label{cor:median}
Let $f: \mathbb{R}^d\rightarrow \mathbb{R}$ be lower semicontinuous and quasiconvex, and define \begin{align*} m^-_{Y} = \inf\{t : \pr\{f(Y) \leq t\}\geq 1/2\}.\end{align*} Under the conditions on $X$ and $Y$ in Corollary \ref{cor:compare}, for any $a > m^-_Y$
\begin{align*} \pr\{f(X) \geq a\} \leq \pr\{f(Y) \geq a\}.\end{align*}  For $a=m^-_Y$, $\pr\{f(X) > m^-_Y\} \leq \pr\{f(Y) > m^-_Y\}$. If additionally $\pr\{f(X)=m^-_Y\}=0$, then $\pr\{f(X) \geq m^-_Y\} \leq \pr\{f(Y) \geq m^-_Y\}$.
\end{corollary}
Corollary \ref{cor:median} proves the conjecture of \citet{ehm18} under weaker conditions than convex monotonicity. It also establishes the tail ordering proposed in the conjecture of \citet{coh22} under the stronger Loewner covariance condition. In the appendix, we present a counterexample showing that the original conjecture of \citet{coh22} under marginal variance dominance and the Sudakov–Fernique increment condition is false.

\section{Proofs of Theorem \ref{thm:conv} and Corollaries \ref{cor:compare}--\ref{cor:median}}\label{sec:proof}
\subsection{Lemmas for Theorem \ref{thm:conv}}
Before proving Theorem \ref{thm:conv}, we record two useful lemmas.
\begin{lemma}\label{lem:concave}
Let $\mathcal{K}$ be a closed convex set that is a proper subset of $\mathbb{R}^d$ and has a nonempty interior. Let $X\sim N_d(0, \Sigma_X)$ with $\Sigma_X\succ 0$, and define 
\begin{align} G(v) &= \pr(X\in \mathcal{K}+v). \label{eq:G}\end{align}Define $H(v) = \Phi^{-1}\{G(v)\}$, where $\Phi$ and $\Phi^{-1}$ are the cumulative distribution function and quantile function respectively of the standard univariate normal. Then, for any $\lambda \in [0,1]$ and any $s,u \in \mathbb{R}^d$, 
\begin{align}\label{eq:concave}
H \left\{\lambda s + (1-\lambda)u\right\}  &\geq \lambda H(s) + (1-\lambda)H(u).
\end{align}
That is, $H(v)$ is concave in $v$.

\end{lemma}
\begin{proof}
As $\Sigma_X \succ 0$, it has a symmetric invertible square root $\Sigma_X = \Sigma_X^{1/2}\Sigma_X^{1/2}$. Then, 
\begin{align*} G(v) = \gamma_d\{\Sigma_X^{-1/2}\left(\mathcal{K} + v\right)\},\end{align*} where $\gamma_d$ is the $d$-dimensional Gaussian measure with mean zero and identity covariance matrix. 

As $\mathcal{K}$ is convex, we have for any $s, u \in \mathbb{R}^d$ and for $\lambda \in [0,1]$ that \begin{align*} \Sigma_X^{-1/2}\left\{\mathcal{K} + \lambda s + (1-\lambda)u\right\} &= \Sigma_X^{-1/2}\left\{\lambda(\mathcal{K} + s) + (1-\lambda)(\mathcal{K} + u)\right\}\\ &= \lambda \Sigma_X^{-1/2}(\mathcal{K}+s) + (1-\lambda)\Sigma_X^{-1/2}(\mathcal{K}+u).\end{align*} Applying the Ehrhard--Borell inequality, we have
\begin{align*}
\Phi^{-1}\left\{G(\lambda s + (1-\lambda)u)\right\} &= \Phi^{-1}\gamma_d\left[\lambda\left\{\Sigma_X^{-1/2}(\mathcal{K} + s)\right\} + (1-\lambda)\left\{\Sigma_X^{-1/2}(\mathcal{K} + u)\right\}\right]\\
& \geq \lambda\Phi^{-1}\left[\gamma_d\left\{\Sigma_X^{-1/2}(\mathcal{K} + s)\right\}\right] + (1-\lambda)\Phi^{-1}\left[\gamma_d\left\{\Sigma_X^{-1/2}(\mathcal{K} + u)\right\}\right]\\
&= \lambda \Phi^{-1}\left\{G(s)\right\} + (1-\lambda)\Phi^{-1}\left\{G(u)\right\}.
\end{align*}
As $\mathcal{K}$ is a proper subset of $\mathbb{R}^d$ with a nonempty interior, $0 < G(v)<1$, and $H(v)=\Phi^{-1}\left\{G(v)\right\} \in \mathbb{R}$ is concave in $v$ as desired.
\end{proof}

\begin{lemma}\label{lem:comparePhi} Let $\Phi$ be the standard normal cumulative distribution function. Then, for $a,b,c\in \mathbb{R}$ such that $a+b \leq 2c$, $c\geq 0$,
\begin{align*}
\Phi(a) + \Phi(b)\leq 2\Phi(c).
\end{align*}
\end{lemma}
\begin{proof} Since $b \leq 2c-a$, by monotonicity of $\Phi$ it suffices to prove the result at $b=2c-a$. At $c=0$, $\Phi(a)+\Phi(-a) = 1 = 2\Phi(0)$. For $c > 0$, define the function $q(a) = \Phi(a) + \Phi(2c-a)$. Differentiating, we have 
\begin{align*}
q'(a) &= \phi(a)-\phi(2c-a),
\end{align*}where $\phi$ is the standard normal density. For $a \leq c$, we have $|a| \leq |2c-a|$, while for $a \geq c$ we have $|a|\geq |2c-a|$. As $\phi$ is symmetric and unimodal at 0, $q'(a) \geq 0$ for $a \leq c$, and $q'(a) \leq 0$ for $a \geq c$. The function $q$ is thus maximized at $a=c$, so that \begin{align*} \Phi(a)+\Phi(b)\leq q(a)\leq q(c) = 2\Phi(c).\end{align*}\end{proof}
\subsection{Proof of Theorem \ref{thm:conv} and relation to Anderson's Theorem}
\begin{proof}
The inequality is trivial when $\mathcal{K} = \mathbb{R}^d$, so suppose that $\mathcal{K}$ is a proper subset of $\mathbb{R}^d$. As $X+W$ is absolutely continuous and $\mathcal{K}$ is convex, $\pr(X+W\in \mathcal{K}) \geq 1/2$ further implies that $\mathcal{K}$ has a nonempty interior. 

For $v\in \mathbb{R}^d$ let $G(v) = \pr(X\in \mathcal{K}+v)$ as in (\ref{eq:G}) and $H(v) = \Phi^{-1}\{G(v)\}$. We first establish that $\pr(X+W \in \mathcal{K})\geq 1/2$ forces $H(0)\geq 0$. Suppose instead that $H(0) < 0$. Applying Lemma \ref{lem:concave} and setting $\lambda=1/2$, $s=v$, and $u=-v$ in (\ref{eq:concave}) for any fixed $v\in \mathbb{R}^d$, this would imply
\begin{align*}
\frac{1}{2}H(v) + \frac{1}{2}H(-v) & \leq H(0) < 0 \Rightarrow H(-v) < -H(v).
\end{align*}
As $\Phi$ is monotone increasing, this would further imply that for any $v\in\mathbb{R}^d$,
\begin{align*}
G(v) + G(-v) &= \Phi\{H(v)\} + \Phi\{H(-v)\}\\
& < \Phi\{H(v)\} + \Phi\{-H(v)\}\\& = \Phi\{H(v)\}+1-\Phi\{H(v)\} = 1.
\end{align*}Using independence of $X$ and $W$ and $W\overset{d}{=}-W$, this would then yield
\begin{align*}
    \pr(X+W\in \mathcal{K}) &= E\left\{\pr(X+W\in \mathcal{K}\mid W)\right\}\\
    &= E\{G(-W)\} = \frac{1}{2}E\{G(W)+G(-W)\} < 1/2,
\end{align*}
contradicting the assumption that $\pr(X+W\in \mathcal{K}) \geq 1/2$. Therefore $H(0)\geq 0$, or equivalently $G(0)\geq 1/2$.

Now define $a = H(v)$, $b=H(-v)$, and $c=H(0)$ for any fixed $v\in\mathbb{R}^d$. We have just established that $c\geq 0$. Again applying Lemma \ref{lem:concave}, we further have $a+b\leq 2c$. Therefore, the conditions of Lemma \ref{lem:comparePhi} hold, so that 
\begin{align*}
\Phi\{H(v)\} + \Phi\{H(-v)\} & \leq 2 \Phi\{H(0)\}.
\end{align*}
Recalling $H(v) = \Phi^{-1}\{G(v)\}$, this implies
\begin{align*}
\frac{1}{2}G(v) + \frac{1}{2}G(-v) &\leq G(0).
\end{align*} Using that $X$ and $W$ are independent and $W\overset{d}{=}-W$, we obtain
\begin{align*}
    \pr(X+W\in \mathcal{K}) &= E\left\{\pr(X+W\in \mathcal{K} \mid W)\right\}\\
    &= E\{G(-W)\} = \frac{1}{2}E\{G(W)+G(-W)\}\\
    &\leq G(0) = \pr(X\in \mathcal{K})
\end{align*}
as desired.
\end{proof}

\begin{remark} If $\mathcal{K}$ is further assumed mirror symmetric, i.e., $\mathcal{K} = -\mathcal{K}$, then $G(v) = G(-v)$ and $H(v) = H(-v)$. Therefore, $H(v) = (1/2)H(v) + (1/2)H(-v)$, and Lemma \ref{lem:concave} then implies 
\begin{align*}
H(v) &= \frac{1}{2}H(v)+\frac{1}{2}H(-v)\\ &\leq H(0),
\end{align*}
so that $G(v) \leq G(0)$ by monotonicity of $\Phi$. For $W$ independent of $X$ (but not necessarily centrally symmetric), conditioning on $W$ then recovers
\begin{align*}
\pr(X+W\in \mathcal{K}) &= E\{\pr(X+W\in \mathcal{K}\mid W)\}\\
&= E\{G(-W)\}\\
&\leq G(0) = \pr(X\in \mathcal{K}),
\end{align*}
the Gaussian convolution consequence of Anderson's Theorem. Theorem \ref{thm:conv} replaces the requirement that $\mathcal{K}$ is mirror symmetric with the conditions that $\pr(X+W\in \mathcal{K})\geq 1/2$ and that $W$ is centrally symmetric. 
\end{remark}

\subsection{Proof of Corollary \ref{cor:compare}.}
\begin{proof}
We break the proof into two cases, depending on whether or not $\Sigma_Y$ is positive definite.

\medskip
\noindent\textit{Case 1: $\Sigma_Y$ is positive definite.}

Let $W \sim \mathcal{N}_d(0, \Sigma_Y-\Sigma_X)$ with $W$ independent of $X$. Observe that $Y \overset{d}{=}X+W$ and that $W\overset{d}{=} -W$.
Suppose first that $\text{rank}(\Sigma_X) = d$. Then, Theorem \ref{thm:conv} implies that if $\pr(Y\in \mathcal{K}) = \pr(X+W \in \mathcal{K})\geq 1/2$, 
\begin{align*} \pr(X\in \mathcal{K})&\geq \pr(X+W\in \mathcal{K})\\ &= \pr(Y\in \mathcal{K}).\end{align*}

Now suppose $\text{rank}(\Sigma_X) < d$. For $t \in [0,1]$, define 
\begin{align*} Z_t &= X + t^{1/2}W.
\end{align*} For all $t\in(0,1]$, $Z_t \sim \mathcal{N}_d(0,\Sigma_t)$ with 
\begin{align*} \Sigma_t &= \Sigma_X+t(\Sigma_Y-\Sigma_X) \succ 0;\\ \Sigma_Y-\Sigma_t &= (1-t)(\Sigma_Y-\Sigma_X) \succeq 0.\end{align*} For each $t\in (0,1]$, let $W_t$ be independent of $Z_t$ with 
\begin{align*}W_t \sim \mathcal{N}_d(0, \Sigma_Y-\Sigma_t),
\end{align*}so that $Z_t+W_t \overset{d}{=} Y$ and $W_t \overset{d}{=} -W_t$. The preceding full-rank argument then establishes that for $t \in (0,1]$, whenever $\pr(Y\in \mathcal{K})\geq 1/2$,\begin{align*}\pr(Z_t \in \mathcal{K}) \geq \pr(Y\in\mathcal{K}).\end{align*}Observe that $||X - Z_t||_2^2 = t||W||_2^2$. As $E||W||^2_2 < \infty$, we have that $Z_t$ converges in $L^2$, and hence in distribution, to $X$ as $t\downarrow 0$. As $\mathcal{K}$ is closed, the Portmanteau theorem implies 
\begin{align*}\operatorname{pr}(X\in \mathcal{K}) \geq \limsup_{t\downarrow 0}\operatorname{pr}(Z_t \in \mathcal{K})\geq \operatorname{pr}(Y\in \mathcal{K})\end{align*} as desired.

\medskip
\noindent\textit{Case 2: $\Sigma_Y$ is singular.}

Let $r := \text{rank}(\Sigma_Y)$. If $r=0$ the result is immediate, so suppose $r\geq 1$. Observe that as $0\preceq\Sigma_X\preceq\Sigma_Y$, any $z$ in the null space of $\Sigma_Y$ must be in the null space of $\Sigma_X$, as
\begin{align*} 0\leq z^\T\Sigma_X z\leq z^\T\Sigma_Y z = 0.\end{align*} As $\Sigma_X\succeq 0$ , this implies $\Sigma_Xz = 0$. Thus, by symmetry of $\Sigma_X$ and $\Sigma_Y$, 
\begin{align*}\text{colspace}(\Sigma_X)\subseteq \text{colspace}(\Sigma_Y).\end{align*} Eigendecompose $\Sigma_Y$ as $\tilde{Q}\tilde{\Lambda} \tilde{Q}^\T$. $\tilde{\Lambda}$ is diagonal, with $r$ positive entries and $d-r$ entries equal to zero. Partition $\tilde{Q}=[U,U_0]$ for the eigenvectors corresponding to the $r$ positive and $d-r$ null eigenvalues respectively. Define 
\begin{align*} \tilde{X} = U^\T X \in \mathbb{R}^r;\\\tilde{Y} = U^\T Y \in \mathbb{R}^r,\end{align*} and observe that 
\begin{align*} \Sigma_{\tilde{Y}}\succ 0;\;\;\;\Sigma_{\tilde{Y}} - \Sigma_{\tilde{X}} = U^\T(\Sigma_Y-\Sigma_X)U \succeq 0.\end{align*}Define $\tilde{\mathcal{K}} = \{q\in\mathbb{R}^r: U q \in \mathcal{K}\}$.  Then, since $X, Y \in \text{colspace}(\Sigma_Y)$ a.s., 
\begin{align*}\operatorname{pr}(\tilde{X}\in \tilde{\mathcal{K}}) &= \operatorname{pr}(X\in \mathcal{K});\\ \operatorname{pr}(\tilde{Y}\in \tilde{\mathcal{K}}) &= \operatorname{pr}(Y\in \mathcal{K})\geq 1/2.\end{align*} 
Since $\mathcal{K}$ is closed and convex and $q\mapsto Uq$ is linear, $\tilde{\mathcal{K}}$ is also closed and convex. The proof then follows by applying the Case 1 argument to $\tilde{X}$, $\tilde{Y}$, and $\tilde{\mathcal{K}}$. 
\end{proof}
\subsection{Proof of Corollary \ref{cor:median}}
\begin{proof}
Recall that $m^-_Y = \inf\{t: \pr\{f(Y)\leq t\} \geq 1/2\}$. For any $a \geq m_Y^{-}$, we have 
\begin{align*}\pr\{f(Y)\leq a\} \geq 1/2.\end{align*} The case $a > m^-_Y$ follows by definition of $m^-_Y$, while the case $a=m^-_Y$ follows by continuity from above of probability measures. 

As $f$ is quasiconvex and lower semicontinuous by assumption, the sublevel set \begin{align*}\mathcal{K}_a = \{w: f(w)\leq a\}\end{align*} is closed (by lower semicontinuity) and convex (by quasiconvexity), and for $a\geq m^-_Y$ we have established
\begin{align*}
\pr(Y\in \mathcal{K}_a) = \pr\left\{f(Y)\leq a\right\} \geq 1/2.
\end{align*}Applying Corollary \ref{cor:compare}, we have for every $a\geq m^-_Y$
\begin{align*}
\pr\{f(X)\leq a\}\geq \pr\{f(Y)\leq a\},
\end{align*}
or equivalently
\begin{align} \label{eq:strict}
\pr\{f(X) > a\} \leq \pr\{f(Y) > a\}.
\end{align}To obtain the comparison with non-strict inequalities for $a > m^-_Y$, apply (\ref{eq:strict}) for $b$ satisfying $m^-_Y\leq b < a$ and send $b\uparrow a$. As $\{f > b\}\downarrow \{f\geq a\}$ for $b\uparrow a$, continuity from above of probability measures gives
\begin{align*}
\pr\{f(X) \geq a\} \leq \pr\{f(Y) \geq a\}.
\end{align*}For $a = m^-_Y$, (\ref{eq:strict}) establishes $\pr\{f(X) > m^-_Y\} \leq \pr\{f(Y) > m^-_Y\}$. If $\pr\{f(X) = m^-_Y\}=0$, then
\begin{align*}
\pr\{f(X) \geq m^-_Y\} = \pr\{f(X) > m^-_Y\} \leq \pr\{f(Y) > m^-_Y\} \leq \pr\{f(Y) \geq m^-_Y\}.
\end{align*}
\end{proof}

\section{Illustration and conservative statistical inference}
\subsection{An illustration with max statistics}\label{sec:illustration}

We compare the upper tail probabilities of the maximum absolute value and the maximum, respectively, for two centered Gaussian vectors in dimension two. Although Gaussian maxima are covered by several classical comparison inequalities, none provides the desired upper tail ordering of Corollary \ref{cor:median} under Loewner covariance dominance alone: Sudakov–Fernique compares expected maxima rather than tail probabilities, while Slepian’s inequality requires equal marginal variances. Let $Y$ and $X$ be centered Gaussian vectors with covariance matrices
\begin{align*}
\Sigma_Y = \begin{pmatrix}
   0.06  & 0.05 \\
    0.05  & 1.20
\end{pmatrix};\;\;
\Sigma_X =\begin{pmatrix}
   0.01  & -0.05 \\
    -0.05  & 1.00
  \end{pmatrix};\;\;
 \Sigma_Y-\Sigma_X &= 0.05\begin{pmatrix}1\\2\end{pmatrix}\begin{pmatrix}1 & 2\end{pmatrix}\succeq 0.\end{align*}

Figure \ref{fig:max} plots the upper tail probabilities corresponding to $X$ (dotted) and $Y$ (solid) after the transformations $f(w) = \max(|w_1|, |w_2|)$ (left) and $g(w)= \max(w_1,w_2)$ (right). As $f(w)$ is quasiconvex, continuous, and mirror symmetric, its upper tail probabilities reflect Anderson's theorem: $f(Y)$ stochastically dominates $f(X)$. As $g(w)$ is quasiconvex and continuous but is not mirror symmetric, stochastic dominance no longer holds in the plot on the right. For example,  $\pr\{g(X)\geq 0\} = 0.83 > \pr\{g(Y)\geq 0\} = 0.72$. The upper tail probabilities cross at $a=0.08$, with $\pr\{g(X)\geq a\} \leq \pr\{g(Y)\geq a\}$ for all $a>0.08$. The median of $g(Y)$ is $m^-_Y = \inf\{t:\pr\{g(Y)\leq t\} \geq 1/2\} = 0.23 > 0.08$, reflecting Corollary \ref{cor:median}'s guarantee that the upper tail probabilities of $g(X)$ are bounded by those of $g(Y)$ above the median of $g(Y)$.

\begin{figure}[h]
\begin{center}
\includegraphics[scale=.8]{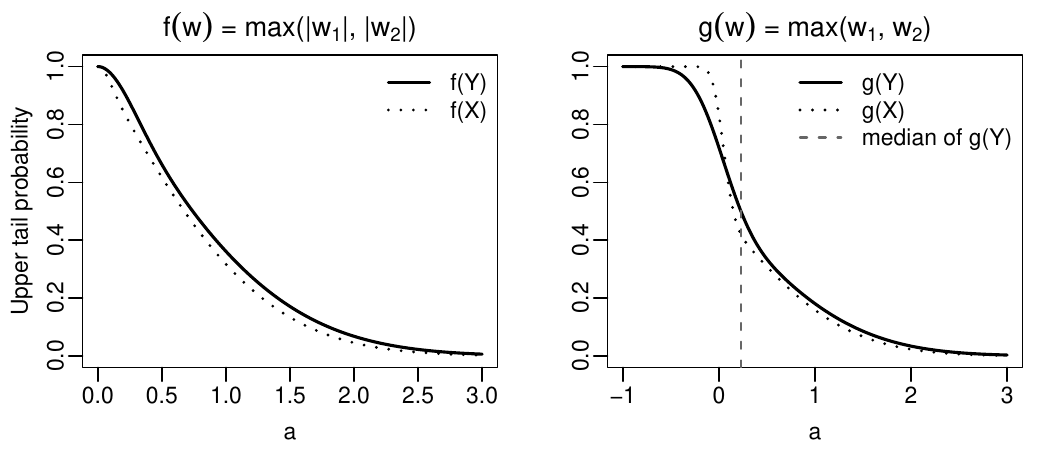}
\caption{Upper tail probabilities for $f(w) = \max(|w_1|, |w_2|)$ (left) and $g(w) = \max(w_1,w_2)$ (right) when applied to $Y$ (solid) and $X$ (dotted), where $Y$ and $X$ are centered multivariate normals with $\Sigma_Y\succeq \Sigma_X$. The plot on the left reflects Anderson's Theorem \citep{and55}, while the plot on the right illustrates Corollary \ref{cor:median}.}
\label{fig:max}
\end{center}
\end{figure}

In the appendix, we demonstrate that changing $E(Y_1^2)
$ from 0.06 to 0.01 produces a setting where $Y$ has marginal variances no smaller than those of $X$ and the Sudakov-Fernique covariance conditions hold, but where $\Sigma_Y \nsucceq \Sigma_X$ and there exists an $a$ above the median of $g(Y)$ such that $\pr\{g(X) \geq a\} > \pr\{g(Y)\geq a\}$. This disproves the original conjecture of \citet[\S 9.1]{coh22} stated under these weaker conditions and highlights the importance of both the condition $\Sigma_Y\succeq \Sigma_X$ in Corollary \ref{cor:compare} and the equal marginal variance assumption in Slepian's inequality.

\subsection{Conservative one-sided statistical inference for $\alpha < 1/2$}
In statistical applications, test statistics are often of the form \begin{align*} T_n = f_{\hat{\xi}}\{n^{1/2}(\hat{\theta}-\theta_0)\},\end{align*} where $\hat{\xi}$ is a nuisance parameter with probability limit $\xi$. For instance, the studentized max statistic $f_{\hat{\Sigma}_Y}\{n^{1/2}(\hat{\theta}-\theta_0)\} = \max_{1\leq i\leq d} n^{1/2}(\hat{\theta}_i-\theta_{0i})/(\hat{\Sigma}_Y)_{ii}^{1/2}$ is often preferred over the nonstudentized variant because it is better balanced in the sense of \citet{ber88balanced}. Furthermore, relevant tail probabilities are calculated using the estimated pushforward measure from a centered multivariate normal with covariance $\hat{\Sigma}_Y$. The following proposition justifies conservative inference using conservative covariance estimators for $\alpha < 1/2$, accommodating both nuisance parameter estimation and plug-in estimation of the upper tail probabilities.

\begin{proposition}\label{prop:inference}
Let $f_u:\mathbb{R}^d\to \mathbb{R}$ be quasiconvex with $f_u(w)$ jointly continuous in $(u,w)$. Let $S_u(t;\Sigma) = \pr\{f_u(W)\geq t\}$ where $W\sim\mathcal{N}_d(0, \Sigma)$. Suppose that $n^{1/2}(\hat{\theta}-\theta_0)$ converges in distribution to $X\sim \mathcal{N}_d(0, \Sigma_X)$, that $\hat{\Sigma}_Y$ is positive semidefinite and converges in probability to $\Sigma_Y \succeq \Sigma_X$, and that $\hat{\xi}$ converges in probability to $\xi$. Assume that $S_{\xi}(t; \Sigma_Y)$ is continuous in $t$ on $\{t: S_{\xi}(t;\Sigma_Y) < 1/2\}$ and that $\lim_{h\rightarrow 0^+}S_{\xi}(m^+_{Y}+h;\Sigma_Y) = 1/2$, where $m^+_{Y} = \sup\{t:S_\xi(t;\Sigma_Y) \geq 1/2\}.$ Let $T_n = f_{\hat{\xi}}\left\{n^{1/2}(\hat{\theta}-\theta_0)\right\}$. Then, for $\alpha < 1/2$, the upper tail probability $S_{\hat{\xi}}(T_n; \hat{\Sigma}_Y)$ is an asymptotically valid $p$-value:
\begin{align*}
\underset{n\rightarrow\infty}\limsup\;\pr\{S_{\hat{\xi}}(T_n; \hat{\Sigma}_Y)\leq \alpha\}\leq \alpha.
\end{align*}
\end{proposition}

The proof is presented in the appendix, and follows from standard weak convergence and continuous mapping arguments. If $f_\xi(X)$ and $f_\xi(Y)$ are both continuously distributed, the conclusion extends to all $\alpha\leq 1/2$ by a simpler continuous mapping argument than that used to prove Proposition \ref{prop:inference}. That said, statistics producing discontinuous reference upper tail probabilities $S_\xi(t;\Sigma_Y)$ are not uncommon in multivariate one-sided and order-restricted inference. For example, chi-bar-square laws may contain atoms at the lower endpoint 0, attained whenever the projection onto the cone of interest equals the origin. Proposition \ref{prop:inference} accommodates atoms below the relevant upper tail region, including the atoms at zero when present in chi-bar-square laws.

\section*{Acknowledgements}
The author would like to thank the Associate Editor and the Editor for their constructive comments, which substantially improved the quality of the paper. The author is especially grateful to the Associate Editor for suggesting the stronger formulation in Theorem \ref{thm:conv} and its proof. Theorem \ref{thm:conv} is more general than the author's original result (now stated as Corollary \ref{cor:compare}) and provides a simpler argument for that result than the author's original proof.

\medskip

\noindent The author was supported by NSF Grant DMS-2413484.

\section*{Generative AI and AI-assisted technologies}

During the preparation of this work, the author used OpenAI's ChatGPT to help formulate the numerical example in \S \ref{sec:illustration} and the counterexample to the conjecture of \citet[\S 9.1]{coh22} in the appendix. It was also used to check algebraic and proof details. After using this tool, the author reviewed and edited the content as necessary and takes full responsibility for the content of the publication.

\appendix
\section{Counterexample under marginal variance dominance and Sudakov-Fernique conditions}
We now present a counterexample disproving the conjecture of \citet[\S 9.1]{coh22} which replaced the condition $\Sigma_Y\succeq \Sigma_X$ with a combination of weakly larger marginal variances and the Sudakov-Fernique increment conditions. Let $Y$ and $X$ be centered Gaussians with covariance matrices
\begin{align*}
\Sigma_Y = \begin{pmatrix}
   0.01  & 0.05 \\
    0.05  & 1.20
\end{pmatrix};\;\;
\Sigma_X =\begin{pmatrix}
   0.01  & -0.05 \\
    -0.05  & 1.00
  \end{pmatrix}.\end{align*}
The eigenvalues of $\Sigma_Y-\Sigma_X$ are 0.24 and -0.04, so the difference is indefinite and the requirement in Corollary 2.2 and Corollary 2.3 that $\Sigma_Y\succeq \Sigma_X$ does not hold. That said, we have $E(Y_1^2)-E(X_1^2)=0$, $E(Y_{2}^2)-E(X_2^2) = 0.2$, and $E\{(Y_1-Y_2)^2\} - E\{(X_1-X_2)^2\} = 0$. Therefore, the marginal variances of $Y$ are no smaller than those of $X$, and $Y$ satisfies equality in the Sudakov-Fernique conditions, so the conditions of the conjecture in \citet[\S 9.1]{coh22} hold. 
\begin{figure}[h]
\begin{center}
\includegraphics[scale=.8]{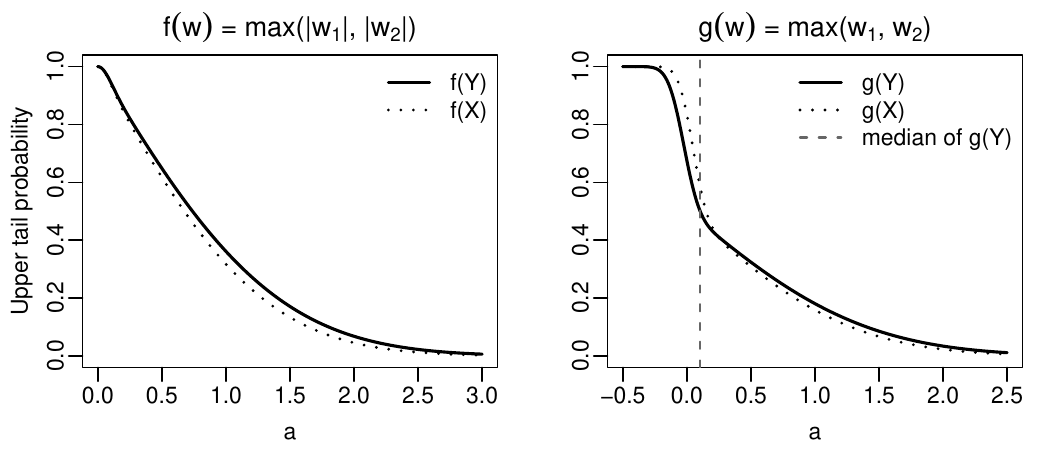}
\caption{Upper tail probabilities for $f(w) = \max\{|w_1|, |w_2|\}$ (left) and $g(w) = \max\{w_1,w_2\}$ (right) when applied to $Y$ (solid) and $X$ (dotted), where $Y$ and $X$ are centered multivariate normals such that $Y$ has weakly larger marginal variances than $X$, $Y$ and $X$ satisfy the covariance conditions of the Sudakov-Fernique inequality, but $\Sigma_Y\nsucceq \Sigma_X$.}
\label{fig:sudakov}
\end{center}
\end{figure}

Figure \ref{fig:sudakov} shows the upper tail probabilities for $f(w) = \max(|w_1|, |w_2|)$ (left) and $g(w) = \max(w_1,w_2)$ (right). The plot on the right shows that the upper tail probabilities cross above the median. Under the stronger condition $\Sigma_Y\succeq \Sigma_X$, Corollary 2.3 rules out any crossing of the upper tail curves above the median of $g(Y)$. The median of $g(Y)$ is $m^{-}_Y=0.10$, while the crossing occurs at 0.23, roughly the 60th percentile. At $a=0.11$, above the median of $g(Y)$, we have $\operatorname{pr}\{g(X)\geq 0.11\} = 0.57$, while $\operatorname{pr}\{g(Y)\geq 0.11\} = 0.49$.

\section{Proof of Proposition 4.1}
\begin{proof}
By the continuous mapping theorem, joint continuity of $f_u(w)$ in $(u, w)$ implies that if $\hat{\xi}$ converges in probability to $\xi$ and $n^{1/2}(\hat{\theta}-\theta_0)$ converges in distribution to $X \sim \mathcal{N}_d(0, \Sigma_X)$, then 
\begin{align*} T_n = f_{\hat{\xi}}\{n^{1/2}(\hat{\theta}-\theta_0)\} \overset{d}{\rightarrow} f_\xi(X).\end{align*} Given $\hat{\Sigma}_Y$ and $\hat{\xi}$, let 
\begin{align*} Y^*\mid \hat{\Sigma}_Y,\hat{\xi} \sim \mathcal{N}_{d}(0, \hat{\Sigma}_Y).\end{align*} We now show that the conditional distribution $f_{\hat{\xi}}(Y^*)\mid \hat{\Sigma}_Y, \hat{\xi}$ converges weakly in probability to $f_{\xi}(Y)$. It suffices to show that for all bounded continuous functions $\varphi$,
\begin{align*} E[\varphi\{f_{\hat{\xi}}(Y^*)\}\mid \hat{\Sigma}_Y, \hat{\xi}] &\overset{p}{\rightarrow} E[\varphi\{f_{{\xi}}(Y)\}],\end{align*} where $Y\sim \mathcal{N}_d(0,\Sigma_Y)$ and $\Sigma_Y\succeq \Sigma_X$. 

We proceed by first establishing continuity of the map 
\begin{align*} (u, V)\mapsto E[\varphi\{f_{u}(V^{1/2}Z)\}] \end{align*} for $Z\sim \mathcal{N}_d(0, I_d)$ and $V\succeq 0$. For any sequence $V_m\rightarrow V$ with $V_m\succeq 0$ for all $m$ and any fixed $z$, \begin{align*} V_m^{1/2}z\rightarrow V^{1/2} z\end{align*} by continuity of the symmetric square root. For any sequence $u_m\rightarrow u$, \begin{align*}f_{u_m}(V_m^{1/2}z)\rightarrow f_{u}(V^{1/2}z)\end{align*} by joint continuity of $f_u(w)$ in $(u, w)$. Therefore, by continuity of $\varphi$, $\varphi\{f_{u_m}(V_m^{1/2}z)\}$ converges for any fixed $z$ to $\varphi\{f_{u}(V^{1/2}z)\}$. Boundedness of $\varphi$ and dominated convergence then imply that \begin{align*} E[\varphi\{f_{u_m}(V_m^{1/2}Z)\}]\rightarrow E[\varphi\{f_{u}(V^{1/2}Z)\}]\end{align*} for $Z\sim \mathcal{N}_d(0, I_d)$, so the map $(u, V)\mapsto E[\varphi\{f_{u}(V^{1/2}Z)\}]$ is continuous in $(u, V)$. 

Let $Z\sim \mathcal{N}_d(0, I_d)$ be independent of $(\hat{\xi}, \hat{\Sigma}_Y)$. Since 
\begin{align*}(\hat{\xi}, \hat{\Sigma}_Y) \overset{p}{\rightarrow}(\xi, \Sigma_Y),\end{align*} the continuous mapping theorem gives 
\begin{align*}E[\varphi\{f_{\hat{\xi}}(Y^*)\}\mid \hat{\Sigma}_Y, \hat{\xi}] &= E[\varphi\{f_{\hat{\xi}}(\hat{\Sigma}_Y^{1/2}Z)\}\mid \hat{\Sigma}_Y, \hat{\xi}]\\ & \overset{p}{\rightarrow} E[\varphi\{f_{{\xi}}(\Sigma_Y^{1/2}Z)\}] = E[\varphi\{f_{{\xi}}(Y)\}]\end{align*} as desired. 

Recall that $S_u(t;\Sigma) = \pr\{f_u(W)\geq t\}$ where $W\sim\mathcal{N}_d(0, \Sigma)$. The upper tail probabilities of $f_{\hat{\xi}}(Y^*)\mid \hat{\Sigma}_Y, \hat{\xi}$ are thus $S_{\hat{\xi}}(t;\hat{\Sigma}_Y)$. By the established weak convergence in probability of the conditional law $f_{\hat{\xi}}(Y^*)\mid \hat{\Sigma}_Y, \hat{\xi}$ to $f_{\xi}(Y)$, 
\begin{align*} S_{\hat{\xi}}(t;\hat{\Sigma}_Y)\overset{p}{\rightarrow} S_{\xi}(t;\Sigma_Y)\end{align*} at all continuity points of $S_{\xi}(t;\Sigma_Y)$. Fix $\alpha < 1/2$ and choose $\epsilon > 0$ such that $\alpha+\epsilon < 1/2$. Recall that by assumption, $S_{\xi}(t; \Sigma_Y)$ is continuous in $t$ on $\{t: S_{\xi}(t;\Sigma_Y) < 1/2\}$ and that $\lim_{h\rightarrow 0^+}S_{\xi}(m_Y^+ +h;\Sigma_Y) = 1/2$, where $m_Y^+ =  \sup\{t:S_\xi(t;\Sigma_Y) \geq 1/2\}$. This implies that there exists a continuity point $c$ such that 
\begin{align*}\alpha < S_{\xi}(c;\Sigma_Y) < \alpha+\epsilon.\end{align*} Then, weak convergence in probability implies 
\begin{align*}S_{\hat{\xi}}(c;\hat{\Sigma}_Y) \overset{p}{\rightarrow} S_{\xi}(c;\Sigma_Y) > \alpha,\end{align*} and hence that 
\begin{align*}
\operatorname{pr}\{S_{\hat{\xi}}(T_n;\hat{\Sigma}_Y) \leq \alpha, T_n \leq c\} \rightarrow 0,
\end{align*}
so that $S_{\hat{\xi}}(T_n;\hat{\Sigma}_Y) \leq \alpha$ implies $T_n > c$ with probability tending to one. 

Since $T_n$ converges in distribution to $f_\xi(X)$ where $X\sim \mathcal{N}_d(0, \Sigma_X)$, and $S_{\xi}(c;\Sigma_Y) < 1/2$ implies $c$ is above $m^-_Y = \inf\{t: \operatorname{pr}\{f_\xi(Y)\leq t\} \geq 1/2\}$, Corollary 2.3 in the main text implies
\begin{align*}
    \limsup_{n\rightarrow\infty} \;\operatorname{pr}\left\{S_{\hat{\xi}}(T_n;\hat{\Sigma}_Y) \leq \alpha\right\} \leq \operatorname{pr}\{f_\xi(X) \geq c\} \leq \operatorname{pr}\{f_\xi(Y) \geq c\} = S_{\xi}(c;\Sigma_Y) < \alpha+\epsilon.
\end{align*}
Sending $\epsilon\rightarrow 0$ completes the proof.
\end{proof}

\bibliographystyle{apalike}
\bibliography{bibliography}
\end{document}